\documentclass[11pt]{amsart}
\usepackage{lmodern}
\usepackage{amsmath, amsthm, amssymb}
\usepackage{amsfonts}
\usepackage[normalem]{ulem}
\usepackage{hyperref}

\usepackage[utf8]{inputenc}

\usepackage[margin=3.5cm]{geometry}

\usepackage{verbatim} 
\usepackage{longtable}
\usepackage{faktor}
\usepackage{mathtools}

\usepackage{tikz}
\usetikzlibrary{decorations.pathmorphing}
\tikzset{snake it/.style={decorate, decoration=snake}}

\usepackage{caption}
\usepackage{tikz-cd}
\usetikzlibrary{arrows}

\usepackage{diagbox}
\usepackage{booktabs}

\usepackage{pdflscape}

\DeclareSymbolFont{bbold}{U}{bbold}{m}{n}
\DeclareSymbolFontAlphabet{\mathbbold}{bbold}

\newtheoremstyle{mytheorem}
{3pt}
{3pt}
{\itshape}
{}
{\bf}
{.}
{.5em}
{}

\theoremstyle{mytheorem}

\newtheorem{thm}{Theorem}[section]
\newtheorem{cor}[thm]{Corollary}
\newtheorem{lem}[thm]{Lemma}

\theoremstyle{definition}

\newtheorem{example}[thm]{Example}

\theoremstyle{remark}
\newtheorem{rmk}[thm]{Remark}

\newcommand{\BR}{{\mathbb{R}}}

\newcommand{\CC}{{\mathbb C}}

\newcommand{\QQ}{\mathbb{Q}}

\DeclareFontFamily{OT1}{rsfs}{}
\DeclareFontShape{OT1}{rsfs}{n}{it}{<-> rsfs10}{}
\DeclareMathAlphabet{\curly}{OT1}{rsfs}{n}{it}

\newcommand{\FP}{\ZZ_p}
\newcommand{\An}{A_n^\bullet}

\newcommand{\Part}{\operatorname{Comb}}
\newcommand{\Comb}{\operatorname{Comb}}

\newcommand{\Z}{\mathbb{Z}}
\newcommand{\D}{\operatorname{D}}

\newcommand{\Ins}{\operatorname{Ins}}

\newcommand{\ZZ}{\mathbb{Z}}

\newcommand{\Perm}{\operatorname{Perm}}

\newcommand{\Coker}{\operatorname{Coker}}
\newcommand{\Ker}{\operatorname{Ker}}
\newcommand{\rk}{\operatorname{rk}}
\newcommand{\im}{\operatorname{Im}}

\newcommand{\size}{\operatorname{size}}
\newcommand{\ind}{\operatorname{ind}}

\newcommand\mydots{,\makebox[1em][c]{.\hfil.\hfil.},}
\newcommand\lmydots{,\makebox[1em][c]{.\hfil.\hfil.}}
\newcommand\rmydots{\makebox[1em][c]{.\hfil.\hfil.},}

\newcommand{ \myring}{ \Lambda \langle x_1, x_2 \mydots y_0, y_1 \lmydots \rangle}

\newcommand{\plh}{%
  {\ooalign{$\phantom{0}$\cr\hidewidth$\scriptstyle\times$\cr}}%
}

\usepackage{setspace}

\begin{document}

\title[Integral Cohomology of Configuration Spaces of the Sphere]{Integral Cohomology of Configuration Spaces of the Sphere}
\date{\today}

\author{Christoph Schiessl}
\address{ETH Z\"urich, Department of Mathematics}
\email{christoph.schiessl@math.ethz.ch}

\thanks{ The author was supported by the grant ERC-2012-AdG-320368-MCSK.
	I want to thank Frederik Cohen, Emanuele Delucchi, Emmanuel Kowalski, Paolo Salvatore, Johannes Schmitt, Junliang Shen for very helpful discussions and especially Rahul Pandharipande for his invaluable support.
	This is part of the author's PhD thesis.
}

\begin{abstract}
We compute the cohomology of the unordered configuration spaces of the sphere $S^2$ with integral and with $\ZZ/p \ZZ$-coefficients using a cell complex by Fuks, Vainshtein and Napolitano.
\end{abstract}

\baselineskip=14.5pt

\maketitle

\tableofcontents

\section{Introduction}

\subsection{Representation Stability}

For any topological space $X$, let \[ F_n(X) = \{ (x_1, \dots, x_n) \in X^n | \,  x_i \neq x_j \}\] be the ordered configuration space of $n$ distinct points in $X$. The symmetric group $S_n$ acts on $F_n(X)$ by permuting the points and the quotient  \[ C_n(X) = F_n(X) / S_n \]  is the unordered configuration space. 
The cohomology of configuration spaces has been widely studied, but is hard to compute explicitly.  

As one of the first examples, Arnold \cite{Arn} showed that the integral cohomology groups of $H^*(C_n(\CC), \ZZ)$ satisfy the following properties: 

        \begin{enumerate} 
                \item (Finiteness) All cohomology groups are finite except $H^0(C_n(\CC), \ZZ)= \ZZ$ and $H^1(C_n(\CC), \ZZ) = \ZZ$ for $n \ge 2$. 
                \item (Vanishing) $H^i(C_n(\CC), \ZZ)=0$ for $i \ge n$. 
                \item (Recurrence)  $H^i(C_{2n+1}(\CC), \ZZ) = H^i(C_{2n}(\CC), \ZZ)$ 
                \item (Stability) For increasing $n$, the cohomology groups stabilize: \[H^i(C_n(\CC), \ZZ) = H^i(C_{n+1}(\CC), \ZZ) \text{ if } n \ge 2i-2.\] The isomorphism is induced by pushing in points from infinity, for example by the map 
                        \[ (z_1, \dots, z_n) \mapsto (z_1, \dots, z_n, 1 + \max |z_i|).\] 
        \end{enumerate}

For rational coefficients, Church \cite[Cor. 3]{church} could prove that 
\[ H^r(C_n(M), \mathbb{Q}) = H^r(C_{n+1}(M), \mathbb{Q}) \text{ if } n>r+1 \] for any connected, orientable manifold $M$ of finite type. This is called \emph{homological stability}. One example is (\cite{sevryuk}, \cite{randal}, \cite{salvatore}):
\[ H^r(C_n(S^2), \mathbb{Q}) = \begin{cases} \QQ & n \ge 3, r = 3 \\ \QQ & n=1, r=2 \\ \QQ & r=0 \\ 0 & \text{otherwise} \end{cases} \]

With integer coefficients however, homological stability turns out to be false in general. For example the computation of $\pi_1 C_n(S^2)$ in \cite[Th. 1.11]{pi} shows that: 
\[ H_1(C_n(S^2), \ZZ) = \ZZ / (2n-2) \ZZ .\] 
With $\ZZ/p\ZZ$-coefficients, homological stability  can be replaced by \emph{eventual periodicity}
\[ H^r(C_n(M), \ZZ/ p \ZZ) = H^r(C_{n+p}(M), \ZZ / p \ZZ) \text{ if } n > 2r \] for any connected manifold $M$ of finite type \cite{Na}, \cite{Can}, \cite{Kup}.

In this paper, we will give an example of this phenomenon by computing the cohomology groups of $C_n(S^2)$ using a cellular complex.
\subsection{ Cohomology of $C_n(\mathbb{C})$}

Let $p$ be a prime. Then Fuks \cite{Fuks} (for $p=2$) and Vainshtein \cite{Vain} gave a combinatorial formula for the groups $H^r(C_n(\mathbb{C}), \FP)$. Define
\[ B_p(n, r) = \left | \left \{ \begin{matrix} 1 \le a_1 \le a_2 \le \dots \le a_g \\   0 \le b_1 < b_2 < \dots < b_h \end{matrix} \middle |  \begin{matrix} \ 2\sum_i p^{a_i}+ 2\sum_j p^{b_j} -  2g-h = r \\ 2\sum_i p^{a_i}+ 2\sum_j p^{b_j}  \le n \end{matrix} \right \} \right |.\] 
They could show that 
\[ \dim H^r(C_n(\mathbb{C}), \ZZ / p \ZZ) = B_p(n,r) .\]

\subsection{Cohomology of $C_n(S^2)$}

Using a cellular decomposition of $C_n(S^2)$ by Napolitano \cite{napolitano}, we compute the cohomology groups of $C_n(S^2)$ with $\ZZ / p \ZZ$-coeffcients in this paper. 
\begin{thm} \label{mainfp}
	Let \[ B'_p(n, r) = \left | \left \{ \begin{matrix} 1 \le a_1 \le a_2 \le \dots \le a_g \\   1 \le b_1 < b_2 < \dots < b_h \end{matrix} \middle |  \begin{matrix} \ 2\sum_i p^{a_i}+ 2\sum_j p^{b_j}+1 -  2g-h = r \\ 2\sum_i p^{a_i}+ 2\sum_j p^{b_j}+2  \le n \\ p \nmid 2( n -2\sum_i p^{a_i}- 2\sum_j p^{b_j} - 1)   \end{matrix} \right \} \right |. \] 
Then 
\[ \dim H^r(C_n(S^2), \ZZ / p \ZZ) = B_p(n,r)+ B_p(n-1, r-2) - B'_p(n, r)- B'_p(n, r-1).\]
\end{thm}

\begin{cor}
	We have
	\[ \dim H^r(C_n(S^2), \ZZ / 2 \ZZ) = B_2(n,r)+ B_2(n-1, r-2).\]
\end{cor}

Eventual periodicity of $H^r(C_n(S^2), \ZZ / p \ZZ)$ can be directly concluded from this description. Theorem \ref{mainfp} could also be deduced from \cite[Th. 18.3]{salvatore}. However, our approach is more elementary and allows to determine the integral cohomology:

\begin{thm} \label{mainzz} The first cohomology groups $H^r(C_n(S_2), \ZZ)$ are 
	\begin{align*} 
		& H^0(C_n(S_2), \ZZ) = \ZZ  & &H^1(C_n(S_2), \ZZ) = 0 \\
		& H^2(C_n(S_2), \ZZ) =  \ZZ / (2n-2) \ZZ &   & H^3(C_n(S_2), \ZZ) = \begin{cases} 0 & n=1,2 \\ \ZZ & n = 3 \\ \ZZ \times \ZZ/2 \ZZ & n \ge 4 \end{cases} \end{align*}
			 For $r \ge 4$, the cohomology groups $H^r(C_n(S^2), \ZZ)$ are finite and contain no elements of order $p^2$. 
\end{thm}

	Hence we can reconstruct all integral cohomology groups by theorem \ref{mainfp} and the universal coefficient theorem. The description of $H^r(C_n(S^2), \ZZ)$ seems to be new.

	We will first explain the computations of the cohomology of $C_n(\mathbb{C})$  with $\ZZ/p \ZZ$-coefficients by Fuks \cite{Fuks} and Vainshtein \cite{Vain} and discuss their cell complex. Afterwards, we present the extension of this cell complex that Napolitano \cite{napolitano} used to calculate $H^*(C_n(S^2), \ZZ)$ for $n \le 9$. The main idea of this paper is the construction of a chain homotopy that simplifies Napolitano's complex.  

\section{Configuration Spaces of the Plane}

\subsection{Conventions}
We write \[ \Comb(n,q) = \{ \, [n_1, \dots, n_q ] \in \Z_{> 0}^q \, | \, n_1 + \dots + n_q = n \} \] for compositions of $n$ into $q$ positive summands, for example
\[ \Comb(5,3) = \{ [3,1,1] ,\, [1,3,1],\, [1,1,3],\,  [2,2,1],\, [2, 1, 2],\, [1,2,2] \, \}.\]
We call $q$ the length and $n$ the size of the composition.

The residue ring $\ZZ / m \ZZ$ is denoted $\ZZ_m$. For any abelian group $G$ and prime $p$, we write $G_{(p)} = \{ g \in G | \,  p^ng =0 \text{ for some } n \}$ for the $p$-torsion subgroup. 
	\subsection{Cellular Decomposition of $\overline{C_n(\mathbb{C})}$} The following construction comes from \cite{Fuks} and \cite{Vain}: The projection \[ \mathbb{C} \to \BR, x+iy \mapsto x \] to the real line maps any configuration in $C_n(\mathbb{C})$ to a finite sets of points in $\BR$. Counting the number of preimages of each of these points, we get a composition of $n$. The union of all points in $C_n(\mathbb{C})$ mapping to the same composition $n = n_1 + \dots + n_s$ and the point $\infty$ is a $n+s$-dimensional cell in the one point compactification $\overline{C_n(\mathbb{C})}$. We denote this  cell by $[n_1, \dots, n_s]$.  All such cells together with the point $\infty$ are a cellular decomposition of $\overline{C_n(\mathbb{C})}$. Using Poincaré-Lefschetz duality for Borel-Moore homology \cite{ginzburg} \cite{vassiliev}\[ H^i(C_n(\mathbb{C})) = \tilde{H}_{2n-i}(\overline{C_n(\mathbb{C})}),\] this cell complex can be used to compute the cohomology of $C_n(\mathbb{C})$. 

	The (co)-chains of the resulting (co)-complex $A_n^{\bullet} = (A_n^{r})_r$ with the property \[ H^*(C_n(\CC), \ZZ ) = H^*(A_n^\bullet)\] are the free $\ZZ$-modules  \[ A_n^r = \Z \Part(n,n-r). \] The basis elements are the compositions $[n_1, \dots, n_s] \in \Comb(n, s)$ with $s= n-r$. The boundary maps $\delta \colon A_n^r \to A_n^{r+1}$ are 

\[ \delta [n_1, \dots, n_s] = \sum_{l=1}^{s-1} (-1)^{l-1} P(n_l,n_{l+1})  [ n_1, \dots, n_{l-1}, n_l +n_{l+1}, n_{l+2}, \dots, n_s ]\] 
where
\[ P(x,y) = \begin{cases} 
	0 & \text{if } x \equiv y \equiv 1 \mod 2, \\
	\dbinom{ \lfloor x/2+y/2 \rfloor}{ \lfloor x/2 \rfloor } & \text{otherwise.}
\end{cases}
\]

\subsection{Cohomology of $C_n(\mathbb{C})$}

As $P(x,y)= 0$ for odd $x$ and $y$, the complex $A_n^\bullet$ can be written as a direct sum \[ A_n^\bullet = A_{n,0}^\bullet \oplus \dots \oplus A_{n,n}^\bullet\] of subcomplexes $A_{n,t}^\bullet$ generated by compositions with $t$ odd entries.

Take any $I \subset \{1, \dots, s+t\}$ with $t$ elements, say $I = \{ i_1, \dots, i_t\}$ where $i_1 < \dots < i_t$. Then we insert $1's$ at the positions $i_1$ to $i_t$ with alternating signs:
\[ \Ins_I [a_1, \dots, a_s] = (-1)^{ \sum_j i_j} [a_1, \dots, a_{i_1 -1}, 1, a_{i_1}, \dots, a_{i_2 -2}, 1, a_{i_2-1}, \dots ] \]
The map 
	\[ \Ins_t = (-1)^{st} \sum_{I \subset \{1, \dots, s+t\}, |I| =t} \Ins_I \] 
is actually a chain map \[ \Ins_t \colon A_{n,0}^\bullet  \to A_{n+t,t}^\bullet \] that induces isomorphisms \[ H^r(A_{n-t,0}^\bullet) \simeq H^r(A_{n,t}^\bullet). \]
	Hence we get \[H^*(A_n^\bullet) = H^*(A_{n,0}^\bullet) \oplus H^*(A_{n-1,0}) \oplus \cdots \oplus H^*(A_{0,0}^\bullet). \]
As $A_{n,0}^r = 0$ if $n \equiv 1 \mod 2$ or $n > 2r$, we can immediately deduce the  properties of recurrence and stability of Arnolds description. We write
\[ H^r(C_\infty(\CC)) = \lim_{n \to \infty} H^r(C_n(\CC)). \]

\begin{example}
	The cohomology group $H^0(C_n(\CC), \ZZ) = \ZZ$ is generated by the class of $ (-1)^{n(n-1)/2}[1,\dots, 1 ] = \Ins_n([\,])$. For $n \ge 2$, the cohomology group
	$H^1(C_n(\CC), \ZZ) = \ZZ$ is generated by the class of $[2, 1, \dots, 1] - [1,2, 1, \dots, 1] + \dots = (-1)^{(n-2)(n-3)/2+n} \Ins_{n-2} [2]$.
\end{example}

\subsection{Explicit Basis of $H^*(A_{n,0}^\bullet, \FP)$}
	We will now present the description of the group $H^r(A_{n,0}^\bullet, \FP)$ by Vainshtein and work out some of the details and proofs omitted in \cite{Vain}. In particular, the explicit formula for the base elements is misleading and seems to be wrong in the stated form in \cite{Vain}.

Let $[n_1, \dots, n_s]$ be any composition of $n$. Then the alternating sum of its permutations
\[ \sum_{\sigma \in S_s} \operatorname{sign} (\sigma) [n_{\sigma(1)}, \dots, n_{\sigma(s)}] \] is a cycle in $\An$. With $\FP$-coefficients, the following subset of permutations
	\[ \Perm [n_1, \dots, n_s] = \sum_{\begin{subarray}{l} \sigma \in S_s \text{ where } \sigma(i) < \sigma(j) \\ \text{if } i < j \text{ and } n_i = n_j \text{ or } \\ \text{if } i < j \text{ and } P(n_i,n_j) = 0 \text{ mod } p \end{subarray}}\operatorname{sign} (\sigma) [n_{\sigma(1)}, \dots, n_{\sigma(s)}] \]
creates a cycle in $A_n^\bullet \otimes \FP$. 

Take integers $1 \le i_1 \le  \dots \le i_k$ and $0 \le j_1 < \dots < j_l$ such that \[ m= n -  2 ( p^{i_1} + \dots + p^{i_k} + p^{j_1} + \dots + p^{j_l} ) \ge 0 \] and let  \[r = (2 p^{i_1} -2) + \dots + (2 p^{i_k} -2) +(2 p^{j_1} -1) + \dots + (2 p^{j_l} -1).\] Then we give the chain
	\[ \Ins_m \Perm [2p^{i_1-1}, 2p^{i_1-1}(p-1) \mydots 2p^{i_k-1}, 2p^{i_k-1}(p-1), 2p^{j_1} \mydots 2p^{j_l}] \] the name $x_{i_1} \cdots x_{i_k} y_{j_1} \cdots y_{j_l}$. It is a cycle in $A_{n,m}^r \otimes \FP$ (but not in $\An$ if $k>0$). Vainshtein showed that all such cycles form a basis of $H^r(A_{n}^\bullet, \FP)$. We call the quantity $n-m$ the size of the chain $x_{i_1} \cdots x_{i_k} y_{j_1} \cdots y_{j_l}$.
\begin{thm} \cite{Vain} \cite{cohen} \label{basisfp}
	The ring $H^*(C_\infty, \FP)$ is the free graded commutative algebra over $\FP$ with generators \begin{align*} x_i \text{ for } i \ge 1 & &  \deg(x_i)= 2p^i-2  & & \size(x_i) = 2p^{i} \\ y_i \text{ for }  i \ge 0  & & \deg(y_i) = 2p^i-1 & & \size(y_i) = 2p^i. \end{align*} There is a surjection $H^*(C_\infty(\CC), \FP) \to H^*(C_n(\CC), \FP)$ whose kernel is generated by the monomials $x_{i_1} \cdots x_{i_k} y_{j_1} \cdots y_{j_l}$ such that $\size(x_{i_1} \cdots x_{i_k} y_{j_1} \cdots y_{j_l}) > n$.
\end{thm}

\begin{cor} 
Define
\[ B_p(n, r) = \left | \left \{ \begin{matrix} 1 \le a_1 \le a_2 \le \dots \le a_g \\   0 \le b_1 < b_2 < \dots < b_h \end{matrix} \middle |  \begin{matrix} \ 2\sum_i p^{a_i}+ 2\sum_j p^{b_j} -  2g-h = r \\ 2\sum_i p^{a_i}+ 2\sum_j p^{b_j}  \le n \end{matrix} \right \} \right |. \] 
Hence we have
\[ \dim H^r(C_n(\mathbb{C}), \FP) = B_p(n, r). \]
\end{cor}

\begin{cor} \cite{salvatore}
This can also be written as a generating series:
	\[ \sum_{n,r \ge 0} B_p(n,r) w^r z^n= \frac{1+wz^2}{1-z}\prod_{i>0}\frac{1+w^{2p^i-1}z^{2p^i}}{1-w^{2p^i-2}z^{2p^i}}\]
\end{cor}

\begin{rmk}
The notation suggests a product structure on $H^*(C_\infty(\CC), \FP)$. It comes from the map
	\[ C_n(\CC) \times C_m(\CC) \to C_{n+m}(\CC) \]
by adding the points far apart.
\end{rmk}

\begin{rmk}
As \[ \binom{p^a+p^b}{p^a} \equiv  \begin{cases} 1 & a \neq b \\ 2 & a = b \end{cases} \mod p\] and \[\binom{ p^a + p^b(p-1)}{p^a} \equiv \begin{cases} 1 & a \neq b \\ 0 & a = b \end{cases} \mod p\]
	by Lucas's theorem \cite{lucas}, the order of all entries of the form $2p^a$, $2p^a(p-1)$ in our basis elements is preserved by the operator $\Perm$. 
\end{rmk}
\begin{example} \label{twentyfour}
	We compute $H^*(C_{24}(\CC), \ZZ/3 \ZZ)$. The generators have degrees
\begin{center}
	\begin{tabular}{ccccccc}
		generators &	$x_1$ & $x_2$  & $y_0$ & $y_1$ & $y_2$ & \dots \\
		\hline
		degree & 4 & 16 & 1 & 5 & 17 & \dots \\
		size & 6 & 18 & 2 & 6 & 18 & \dots \\
	\end{tabular}
\end{center}
	In table \ref{basiscoh}, we write down the basis elements and the corresponding chains, however we will omit the application of the $\Ins_t$-operators to lift the chains to sum 24.
	\begin{table}[p]
		\caption{ The cohomology group $H^*(C_{24}(\CC), \ZZ_3)$}
		\label{basiscoh}
		\centering
		\begin{tabular}{ c | p{13cm}}
			$r$ & basis of $H^r(C_{24}(\CC), \ZZ_3)$ \\
			\toprule
			0 & $1 = []$ \\
			\midrule
			1 & $y_0 = [2]$ \\
			\midrule
			2 & --\\
			\midrule
			3 & --\\
			\midrule
			4 & $x_1 = [2,4]$ \\
			\midrule
			5 & $y_1 = [6]$ \newline  $x_1 y_0 = [2,4,2]$ \\
			\midrule
			6 & $y_0 y_1 = [2,6] -[6,2]$ \\
			\midrule
			7 & -- \\
			\midrule
			8 & $x_1^2 = [2,4,2,4] $\\
	\midrule
			9 & $x_1 y_1 = [2,4, 6] - [2,6,4] + [6,2,4]$ \newline $x_1^2 y_0 = [2,4,2,4,2]$ \\
	\midrule
			10 & $x_1 y_0 y_1 = [2,4, 2,6] - [2,4,6,2] + [2,6,4,2] - [6,2,4,2] $ \\
			\midrule
			11 & --\\
	\midrule
			12 & $x_1^3 = [2,4,2,4,2,4]$ \\
			\midrule
			13 & $x_1^2 y_1 = [2,4,2,4,6] - [2,4,2,6,4] + [2,4,6,2,4] - [2,6,4,2,4] + [6,2,4,2,4]$ \newline $x_1^3 y_0 = [2,4,2,4,2,4,2]$ \\
		\midrule
		14 & $x_1^2 y_0 y_1 =	[2,4,2,4,2, 6] - [2,4,2,4,6,2] + [2,4,2,6,4,2] - [2,4,6,2,4,2,] + \dots$ \\ 
			\midrule
	15 & --\\
				\midrule
16 & $x_2 = [6, 12]$ \newline $x_1^4 = [2,4,2,4,2,4,2,4]$ \\
			\midrule
			17 & $y_2 = [18]$ \newline $x_2 y_0 = [6,12,2] - [6,2, 12] + [2, 6,12]$ \newline $x_1^3 y_1 = [2,4,2,4,2,4,6] - [2,4,2,4,2,6,4]+ \dots$ \\
				\midrule
18 & $y_0 y_2 = [2, 18] - [18,2]$ \\
		\midrule
		19 & -- \\
			\midrule
	20 & $x_1 x_2 = [2,4, 6,12] - [2,6,4, 12] + [6,2, 4,12] - [6,2, 12, 4] + [2,6,12,4] +  [6,12, 2, 4]$ \\
			\midrule
	21 & $x_1 y_2 = [2,4, 18] - [2,18,4] + [18, 2,4]$ \newline $x_2 y_1 = [6,12, 6]$ \\
		\midrule
		22 & $y_1 y_2 = [6,18] -[18,6]$ \\
			\midrule
	$\ge 23$ & -- \\
		\end{tabular}
\end{table}
\end{example}

%
%

\subsection{Bockstein Homomorphisms}
The short exact sequences of coefficients
\[ 0 \to \ZZ \xrightarrow{p \cdot} \ZZ \to \ZZ_p \to 0 \]
and
\[ 0 \to \ZZ_p \xrightarrow{p \cdot} \ZZ_{p^2} \to \ZZ_p \to 0 \]
induce long exact sequences 
\[ H^{r-1}(\An, \FP) \xrightarrow{\tilde{\beta}} H^r(\An, \ZZ) \xrightarrow{p \cdot}  H^r(\An, \ZZ) \to H^r(\An, \FP) \xrightarrow{\tilde{\beta}} H^{r+1}(\An, \ZZ) \]
and
\[ H^{r-1}(\An, \FP) \xrightarrow{\beta} H^r(\An, \FP) \xrightarrow{p \cdot}  H^r(\An, \ZZ_{p^2}) \to H^r(\An, \FP) \xrightarrow{\beta} H^{i+1}(\An, \FP), \]
where the connecting morphisms are the \emph{Bockstein morphisms} $\beta$ and $\tilde{\beta}$ (compare \cite[Chap.3.E]{hatcherbook}).
The image of $\tilde{\beta}$ are hence all elements of order $p$ in $H^*(\An, \ZZ)$. The following diagram commutes and the upper row is exact:

\begin{tikzpicture}
  \matrix (m) [matrix of math nodes,row sep=3em,column sep=4em,minimum width=2em]
  {
	  H^r(\An, \ZZ) &  H^r(\An, \FP) & H^{r+1}(\An, \ZZ) & H^{r+1}(\An, \ZZ) \\
			& 		&H^{r+1}(\An, \FP) & 		& \\};
  \path[-stealth]
	(m-1-1) edge node [above] {} (m-1-2)
	(m-1-2) edge node [above] {$\tilde{\beta}$} (m-1-3)
	(m-1-3) edge node [above] {$p \cdot$} (m-1-4)
	(m-1-3) edge node [right] {} (m-2-3)
	(m-1-2) edge node [below] {$\beta$} (m-2-3)
	;
\end{tikzpicture}
\begin{example} Let $i \neq j$. We determine the Bockstein on $x_{i} = [2p^{i-1}, 2p^{i-1}(p-1)]$ and $x_i y_j=  [2p^{i-1}, 2p^{i-1}(p-1), 2p^j] - [2p^{i-1}, 2p^j, 2p^{i-1}(p-1)] + [2p^j,2 p^{i-1}, 2p^{i-1}(p-1)]$. In $\An$, we get
	\begin{align*} \delta (x_i)  & = \binom{p^{i}}{p^{i-1}} [2p^{i}] = \binom{p^i}{p^{i-1}}y_i \\
	\delta (x_i y_j)  & =  \binom{p^{i}}{p^{i-1}} ([2p^i, 2p^{j}]- [2p^{j}, 2p^i]) = \binom{p^i}{p^{i-1}} y_i y_j \end{align*}
	Hence we can conclude \begin{align*}  \tilde{\beta}(x_i) = \frac{1}{p} \binom{p^{i}}{p^{i-1}} y_i & & \tilde{\beta}(x_i y_j)= \frac{1}{p} \binom{p^{i}}{p^{i-1}} y_i y_j  . \end{align*} 
	The coefficient
		\[ \frac{1}{p} \binom{p^{i}}{p^{i-1}} = \binom{p^{i}- 1}{p^{i-1}-1}\] 
		is an integer congruent to 1 mod $p$ by Lucas' theorem \cite{lucas}
\end{example}

By a similar, a bit tedious computation we get:
\begin{lem} \label{derivation}
The differential $\delta$ on $A_n^\bullet$ operates as follows:
	\[ \delta ( x_1^{a_1} \cdots x_k^{a_k} y_0^{b_1} \dots y_l^{b_l}) = \sum_i \binom{p^i}{p^{i-1}}  x_1^{a_1} \cdots x_i^{a_i-1} \cdots x_k^{a_k} y_i y_0^{b_0} \cdots y_l^{b_l} \]
Hence the Bocksteins are given by 
	\[ \tilde{ \beta } ( x_1^{a_1} \cdots x_k^{a_k} y_0^{b_1} \dots y_l^{b_l}) = \frac{1}{p} \sum_i \binom{p^i}{p^{i-1}}  x_1^{a_1} \cdots x_i^{a_i-1} \cdots x_k^{a_k} y_i y_0^{b_0} \cdots y_l^{b_l}\]
and 
\[  \beta  ( x_1^{a_1} \cdots x_k^{a_k} y_0^{b_1} \dots y_l^{b_l}) = \sum_i  x_1^{a_1} \cdots x_i^{a_i-1} \cdots x_k^{a_k} y_i y_0^{b_0} \cdots y_l^{b_l}.\]

\end{lem}

As $\beta^2 = 0$, we can look at the \emph{Bockstein cohomology groups}
\[ BH^*(\An, \FP) = \Ker \beta / \im \beta. \]

\begin{lem} \cite[Cor. 3E.4]{hatcherbook} The group $H^*(\An, \ZZ)$ contains no element of order $p^2$ if  and only if 
	\[ \dim_{\FP} BH^r(\An, \FP) =  \rk H^r(\An, \ZZ). \] In this case the map 
	\[H^*(\An, \ZZ) \to H^*(\An, \FP) \] is injective on the $p$-torsion and its image is $\im \beta$.
\end{lem}

Vainshtein stated that $H^*(\An, \ZZ)$ has no elements of order $p^2$:

\begin{thm} \cite{Vain} \label{mainvain}
	The integral cohomology is given by
	\begin{align*} H^0(C_n(\CC), \ZZ) = \ZZ & & H^1(C_n(\CC), \ZZ) = \ZZ \text{ if } n \ge 2 \end{align*} 
		and
	\[ H^r(C_n(\CC), \ZZ) = \bigoplus_p \tilde{\beta}_p H^{r-1}(C_n(\CC), \FP) \text{ for } r \ge 2.\]
\end{thm}

\begin{proof} Take any $x \in \Ker \beta$ of the form
	\[ x = x_j^k f + x_j^{k-1} y_j g\] for $k \ge 0$, $j>0$ 
	where $f,g$ do not contain $x_j$ or $y_j$. We compute
	\[ \beta(x) = x_j^{k-1} y_j f + x_j^k \beta(f) + x_j^{k-1} y_j \beta(g) .\]
	Hence we see $ \beta(g) = f$ and $\beta( x_j^k g) = x$. So we have shown that
	\[ \Ker \beta / \im \beta = \ZZ_p \otimes \ZZ_p y_0. \qedhere \]
\end{proof}

\begin{rmk} The map $\beta$ looks suspiciously like a derivation. We will first work with integer coefficients. We consider the free graded commutative $\ZZ$-algebra
\begin{align*} \Gamma = \myring & & \deg (x_i) = 2p^i-2 & & \deg(y_i) = 2p^i-1.\end{align*}
with the map
		\[ \beta ( x_1^{a_1} \cdots x_k^{a_k} y_0^{b_1} \dots y_l^{b_l}) = \sum_i x_1^{a_1} \cdots x_i^{a_i-1} \cdots x_k^{a_k} y_i y_0^{b_0} \cdots y_l^{b_l}.\]
Take a copy
	\begin{align*} \Gamma' =  \Lambda \langle X_1, X_2 \mydots Y_0, Y_1 \lmydots \rangle & & \deg (X_i) = 2p^i-2 & & \deg(Y_i) = 2p^i-1.\end{align*}
of $\Gamma$. We can embedded the abelian group $\Gamma$ into $\Gamma' \otimes \QQ$ via 
	\[ \Gamma \xhookrightarrow{} \Gamma' \otimes \QQ, \quad x_1^{a_1} \cdots x_k^{a_k} y_0^{b_1} \dots y_l^{b_l} \mapsto \frac{1}{a_1!} X_1^{a_1} \cdots  \frac{1}{a_k!} X_k^{a_k} Y_0^{b_1} \dots Y_l^{b_l}. \] Write $\star$ for the multiplication on $\Gamma'$. Then \[ x_i^{j_1} \star x_i^{j_2} = \binom{j_1+j_2}{j_1} x_i^{j_1+j_2}\] and $\star$ induces a multiplication on $\Gamma$ (a so called divided power algebra \cite[Ex 3.5C]{hatcherbook}). The advantage of $\star$ is that the map $\beta = \beta'_{| \Gamma}$ comes from the unique derivation $\beta'$ on $\Gamma' \otimes \QQ$ defined by  
\begin{align*} \beta'( X_i) = Y_i & & \beta'(Y_i) = 0 \end{align*}
	and the rule (compare \cite[Chap. 3]{rationalhomotopy})
	\[ \beta'( z_1 \star z_2 ) = \beta'(z_1) \star z_2 + (-1)^{\deg z_1} z_1 \star \beta'(z_2). \]
The Bockstein morphism for $\An$ is now the reduction mod $p$ of $\beta$.	
%
%
\end{rmk}

\begin{cor}
We have an isomorphism
	\[ p \text{-Torsion of } H^{r+1}(C_\infty(\CC), \ZZ) \simeq \text{ degree $r$-part of } \Lambda \langle x_1, x_2, \dots, y_1, y_2, \dots \rangle \otimes \FP.\] for $r > 0$.
\end{cor}

\begin{proof}   Let $R= \Lambda \langle x_1, x_2, \dots, y_1, y_2, \dots \rangle \otimes \FP$. 
	Theorem \ref{basisfp} shows that 
	\[ H^* ( C_\infty(\CC), \FP) = R \oplus y_0 R.\]
	By lemma \ref{derivation} we know that $\beta(x y_0) = \beta(x) y_0$ and $\beta(R) \subset R$. This shows
	\[ \im \beta = \beta(R) \oplus y_0 \beta(R).\]
	Decompose $ R = \beta(R) \oplus R'$. As $\Ker \beta = \im \beta \oplus \FP \oplus \FP y_0$, the map
	\[ \beta(R) \oplus R' \to \beta(R) \oplus y_0 \beta(R) = \im \beta , \, (z_1, z_2) \mapsto \beta(z_2)+ y_0 z_1 \] is a bijective map between the degree $r$ part of $R$ and the degree $r+1$ part of $\im \beta$ for $r >0$ . However, it does not respect the size, so the isomorphism is only possible for $n \to \infty$.
\end{proof}

\begin{rmk}
	The description of dimension of the $p$-torsion of $H^r(C_n(\CC), \ZZ)$ in \cite[Appendix to III]{cohen} seems to be wrong.
\end{rmk}

\begin{example}
	In table \ref{torsion}, we compute $H^*(C_{24}(\CC), \ZZ_3)_{(3)}$ by applying theorem \ref{mainvain} and the formula \ref{derivation} to our example \ref{twentyfour}.

\begin{table}[h]
	\caption{ The 3-torsion in the cohomology group $H^*(C_{24}(\CC), \ZZ)$}
	\label{torsion}
	\begin{center}
		\begin{tabular}{ c | p{13cm}}
		$r$ & basis of $H^r(C_{24}(\CC), \ZZ)_{(3)}$ as $\ZZ_3$-module \\
		\toprule
		0 & -- \\
		\midrule
		1 & -- \\
		\midrule
		2 & --\\
		\midrule
		3 & --\\
		\midrule
		4 & --  \\
		\midrule
		5 & $y_1 = [6]$  \\
		\midrule
		6 & $y_0 y_1 = [2,6] -[6,2]$ \\
		\midrule
		7 & -- \\
		\midrule
		8 & \\
		\midrule
		9 & $x_1 y_1 = [2,4, 6] - [2,6,4] + [6,2,4]$ \\
		\midrule
		10 & $x_1 y_0 y_1 = [2,4, 2,6] - [2,4,6,2] + [2,6,4,2] - [6,2,4,2] $ \\
		\midrule
		11 & --\\
		\midrule
		12 & --\\
		\midrule
		13 & $x_1^2 y_1 = [2,4,2,4,6] - [2,4,2,6,4] + [2,4,6,2,4] - [2,6,4,2,4] + [6,2,4,2,4]$ \\
		\midrule
		14 & $x_1^2 y_0 y_1 =	[2,4,2,4,2, 6] - [2,4,2,4,6,2] + [2,4,2,6,4,2] - [2,4,6,2,4,2,] + \dots$ \\ 
		\midrule
		15 & --\\
		\midrule
		16 & -- \\
		\midrule
		17 & $y_2 = [18]$ \newline $x_1^3 y_1 = [2,4,2,4,2,4,6] - \dots$ \\
		\midrule
		18 & $y_0 y_2 = [2, 18] - [18,2]$ \\
		\midrule
		19 & -- \\
		\midrule
		20 & -- \\
		\midrule
			21 & $28 x_1 y_2 + x_2 y_1 = 28([2,4, 18] - [2,18,4] + [18, 2,4]) + [6,12, 6]$ \\
		\midrule
		22 & $y_1 y_2 = [6,18] -[18,6]$ \\
		\midrule
		$\ge 23$ & -- \\
		\end{tabular}
	\end{center}
\end{table}
\end{example}

\newpage
\section{Configuration Spaces of the Sphere}

We will describe a cellular decompostion of $\overline{C_n(S^2)}$ by Napolitano \cite{napolitano} and show how it can be used to compute the cohomology of $C_n(S^2)$.

\subsection{Cellular Decomposition of $\overline{C_n(S^2)}$}
Using $S^2 = \BR^2 \sqcup \infty$, the cellular decomposition of $\overline{C_n(\mathbb{C})}$ can be extended to a cellular decomposition of $\overline{C_n(S^2)}$ by looking at configurations that do or do not contain $\infty$. The resulting complex $B_n^\bullet = (B_n^r)$ with $H^*(B_n^\bullet, \ZZ) = H^*(C_n(S^2), \ZZ)$ has chains
\[ B_n^{r} = A_n^r \oplus A_{n-1}^{r-2} =  \ZZ \Part(n,n-r) \oplus \ZZ \Part(n-1, n-r+1).\] The new boundary maps $\Delta$ were computed by Napolitano \cite{napolitano}: We define a new operator $D \colon A_n^r \to A_{n-1}^{r-1}$ by 
\[ \D [n_1, \dots, n_s] = \sum_{i=1}^s Q(n_i) (-1)^{\sum_{j=1}^{i-1} n_i} [n_1, \dots, n_{i-1}, n_i -1, n_{i+1}, \dots, n_s]\] where
\[ Q(n_i ) = \begin{cases} 
	0 & \text{if } n_i \equiv 1 \mod 2 \\
	2 & \text{ otherwise}.
\end{cases}
\]
The differential $\Delta$ of the complex $B_n^\bullet$ is then given by \[ \Delta \colon B_n^{r} \to B_n^{r+1}, (a,b) \mapsto (\delta(a), \delta(b)+(-1)^{n-r} \D(a) ). \]

\begin{cor} We have $D \equiv 0 \mod 2$ and  therefore $B_n^\bullet \otimes \ZZ_2 =  (A_n^\bullet  \oplus A_{n-1}^\bullet) \otimes \ZZ_2$ and \[ H^r(C_n(S^2), \ZZ_2) = H^r(C_n(\CC), \ZZ_2) \oplus H^{r-2} (C_{n-1}(\CC), \ZZ_2).\]
\end{cor}
\begin{figure}
\begin{tikzpicture}
  \matrix (m) [matrix of math nodes,row sep=3em,column sep=4em,minimum width=2em]
  {
	  \cdots &  A_n^{r-1} & A_n^r & A_n^{r+1} & \cdots \\
	\cdots & A_{n-1}^{r-3} & A_{n-1}^{r-2} & A_{n-1}^{r-1} & \cdots \\};
  \path[-stealth]
	(m-1-1) edge node [above] {$\delta$} (m-1-2)
	(m-2-1) edge node [above] {$\delta$} (m-2-2)
    	(m-1-2) edge node [above] {$\delta$} (m-1-3)
	(m-1-3) edge node [above] {$\delta$} (m-1-4)
	(m-2-2) edge node [above] {$\delta$} (m-2-3)
	(m-2-3) edge node [above] {$\delta$} (m-2-4)
	(m-1-4) edge node [above] {$\delta$} (m-1-5)
	(m-2-4) edge node [above] {$\delta$} (m-2-5)
	(m-1-2) edge node [right] {$S$} (m-2-2)
	(m-1-3) edge node [right] {$S$} (m-2-3)
	(m-1-4) edge node [right] {$S$} (m-2-4)
	(m-1-2) edge node [above] {\hspace{0.5em} $D$} (m-2-3)
	(m-1-3) edge node [above] {\hspace{0.5em} $D$} (m-2-4)
	;

\end{tikzpicture}
\end{figure}

\subsection{Mapping Cone Complex} 

The relation \[D \circ \delta = \delta \circ D\] is equivalent to  $\Delta^2 =0$. This means we can see $D$ as a chain map 
\[ D \colon A_n^\bullet \to A_{n-1}^\bullet [1] \] and the complex $B_n^\bullet$ can be interpreted as the mapping cone complex of the chain map $D$. The short exact sequence of chain complexes 
\[ 0 \to A_{n-1}^\bullet [2] \to B_n^\bullet \to A_n^\bullet \to 0.\]
given by $a_2 \mapsto (0, a_2)$ and $(a_1, a_2) \mapsto a_1$ induces a long exact sequence
\[ \dots \to H^{r-1} (A_n^\bullet) \rightarrow H^r(A_{n-1}^\bullet[2]) \rightarrow H^r(B_n^\bullet) \rightarrow H^r(A_n^\bullet) \rightarrow H^{r+1} (A_{n-1}^\bullet[2]) \rightarrow \dots. \]
The connecting homomorphism can be identified with $D^*$.
\begin{lem} \label{longexact} We get a long exact sequence
	\[ \dots \to H^{r-1} (A_n^\bullet) \xrightarrow{ D^*}  H^{r-2}(A_{n-1}^\bullet) \rightarrow H^r(B_n^\bullet) \rightarrow H^r(A_n^\bullet) \xrightarrow{D^*} H^{r-1} (A_{n-1}^\bullet) \rightarrow \dots \]
\end{lem}
We can use this long exact sequence to compare the cohomology of $B_n^\bullet$, $A_n^\bullet$ and $A_{n-1}^\bullet$. Next we will construct a map 
\[ S \colon A_n^r \to A_{n-1}^{r-2}, \] which is almost a chain homotopy  $D \approx 2 \delta S + 2 S \delta$  between $D$ and the zero map. This allows us to compute the rank of $D^*$.

\section{Construction of (almost) a Null Homotopy}

As a motivation we first look at the case $r = n-1$. We set $S[n]  = [1, n-2]$. Then we have 
\[ 2\delta S [n] = 2 \delta [1, n-2]  = 2 [n-1] = D [n] \] if $n$ is even and 
\[2 \delta S [n] = 2\delta [1, n-2]  = 0 = D [n]\] otherwise.

In general, we define $S \colon A_n^r \to A_{n-1}^{r-2}$ by
 \[ S[n_1 \mydots n_s] = \sum_{1 \le k \le i \le s} \! \! (-1)^{k+1+\sum_{m=1}^{k-1} n_m} [n_1 \mydots n_{k-1}, 1, n_k \mydots n_{i-1}, n_i-2, n_{i+1} \mydots n_s] .\] If $n_i -2 \le 0$, we simply omit this summand.

\begin{lem} \label{E1} For every composition $[n_1, \dots, n_s]$ with $n_s \neq 2$ we have 
	\[ (D-2 \delta \circ S -2 S \circ \delta)  [n_1, \dots, n_s] = 0 \]  and
	\[ (D-2 \delta \circ S - 2 S \circ \delta) [n_1 \mydots n_{s-1}, 2] = \! 2 \sum_{1 \le k \le s} \! \! (-1)^{s+ k+ \sum_{m=1}^{k-1} n_m } [n_1 \mydots n_{k-1}, 1, n_k \mydots n_{s-1}] \] otherwise.  
\end{lem}

\begin{proof}
For convenience we introduce the operators $\delta_l$ by  
	\[ \delta_l [m_1, \dots, m_t] = (-1)^{l-1} P(m_l, m_{l+1}) [m_1, \dots, m_{l-1}, m_{l}+m_{l+1}, m_{l+2}, \dots, m_t] \]
and the abbreviations
	\[n_{k,i} = (-1)^{k+1+\sum_{m=1}^{k-1} n_m} [n_1, \dots, n_{k-1}, 1, n_k, \dots, n_{i-1}, n_i-2, n_{i+1}, \dots, n_s] .\]

	Let us first assume that all $n_i > 2$. We compute \[ \delta \circ S [n_1, \dots n_r] = \sum_{\substack{1 \le l \le s \\ k \le i}} \delta_l (n_{k,i} )\] by splitting up the index set \[I = \{1 \le l \le s, 1 \le k \le i \le s \}\] into \[ I = I_1 \sqcup \dots \sqcup I_8\] where 
	\begin{align*} 
	& I_1 = \{ 1 \le l < k-1, k \le i \} & 			&I_4=  \{ l=i, k<i\} \\ 
	&I_2 = \{k+1 \le l < i \} & 				&I_5 = \{l=i+1, k \le i \} \\ 
	&I_3 = \{ i+2 \le l \le s, k \le i\}& 		&I_6= \{l= k-1, k \le i \}  \\ 
	& & 							& I_7= \{l=k, k < i\} \\ 
	& & 							& I_8= \{l=k=i\}. \end{align*}  
	Now we look at the individual summands $T_j = \sum_{I_j} \delta_l(n_{k,i})$ and expand them after doing some index shifts. Write $\ind = k+l+ \sum_{m=1}^{k-1} n_m$.

\begin{align*} 
	T_1   = & \sum_{\substack{l < k-1 \\ k \le i}} \! \! (-1)^{\ind} P(n_l, n_{l+1}) [ \rmydots n_l+n_{l+1} \mydots n_{k-1}, 1, n_k \mydots n_{i-1}, n_{i}-2, n_{i+1} \lmydots] \\
	T_2  = &  \sum_{ k \le l < i-1} \! \! (-1)^{\ind+1} P(n_l, n_{l+1})[ \rmydots n_{k-1}, 1, n_k \mydots n_{l}+n_{l+1} \mydots n_{i-1}, n_{i}-2, n_{i+1} \lmydots ] \\
	T_3 =  & \sum_{  k \le i < l} (-1)^{\ind+1} P(n_l, n_{l+1})[ \rmydots n_{k-1}, 1, n_k \mydots n_{i-1}, n_{i}-2, n_{i+1} \mydots n_{l}+n_{l+1} \lmydots ] 
	\end{align*}
The next terms 
	\begin{align*}
	 T_4 = & \sum_{k< i} (-1)^{k+i + \sum_{m=1}^{k-1} n_m}  P(n_{i-1}, n_i-2)  [ \rmydots n_{k-1}, 1, n_k \mydots n_{i-1} +n_i-2, n_{i+1} \lmydots] \\
	 T_5 = & \sum_{k \le i} (-1)^{k+i+1+ \sum_{m=1}^{k-1} n_m } P(n_i-2, n_{i+1})  [\rmydots n_{k-1}, 1, n_k \mydots n_{i-1}, n_i-2+n_{i+1} \lmydots] 
	\end{align*}
sum up to
	\[ T_4 + T_5 =  \sum_{k \le i} (-1)^{k+i+1+ \sum_{m=1}^{k-1} n_m } P(n_i, n_{i+1})  [\rmydots n_{k-1}, 1, n_k, \mydots, n_{i-1}, n_i-2+n_{i+1} \lmydots] \]
where we use the identity $P(x-2, y) +P(x, y-2) = P(x,y)$.
Altogether we have
	\[ T_1 +T_2 + T_3 + T_4 + T_5 = - S \circ \delta  [n_1 , \dots, n_s] .\]
The terms 
	\begin{align*}
		T_6 = &  \sum_{k \le  i} (-1)^{2k-2+ \sum_{m=1}^{k-1} n_m} P(n_{k-1}, 1) [ \rmydots n_{k-2}, n_{k-1} +1, n_k \mydots n_{i-1}, n_i -2, n_{i+1} \lmydots] \\
	 T_7 = & \sum_{k < i}  (-1)^{2k-1 +  \sum_{m=1}^{k-1} n_m} P(1, n_k) [\rmydots n_{k-1}, 1 +n_k, n_{k+1} \mydots n_{i-1}, n_i-2, n_{i+1} \lmydots] 
	\end{align*}
cancel each other.
The remaining summand 
	\[ T_8 =   \sum_{i} (-1)^{ \sum_{m=1}^{i-1} n_m } P(1, n_i -2) [ \rmydots n_{i-1}, n_{i}-1, n_{i+1} \lmydots] \]
can be identified with
	\[ 2 T_8 = D [n_1, \dots, n_s] .\]
Here we use $P(1, n_i -2) = 1$ if $n_i$ even and $P(1, n_i-2) = 0$ if $n_i$ odd. In the end we get
	\[2 \delta \circ S [n_1, \dots, n_s] = - 2 S \circ \delta [n_1, \dots, n_s] + D [n_1, \dots, n_s] \]

	In case that $n_j = 2$ with $j<s$, all contributions containing $n_j-2$ in $T_4$, $T_5$ and $T_8$ are missing in $ \delta \circ S$, but not in $S \circ \delta$ and $D$. So we have to add
	\begin{align*}
		T_4' = & \sum_{k< j} (-1)^{k+j + \sum_{m=1}^{k-1} n_m }  P(n_{j-1},0)  [\rmydots 1, n_k \mydots n_{j-2}, n_{j-1}, n_{j+1} \lmydots] \\
	T_5' = &  \sum_{k \le j} (-1)^{k+j+1 + \sum_{m=1}^{k-1} n_m} P(0, n_{j+1})  [\rmydots 1, n_k \mydots n_{j-1}, n_{j+1} \lmydots ] \\
	T_8' = & (-1)^{  \sum_{m=1}^{j-1} n_m}  P(1, 0) [ \rmydots n_{j-1}, 1, n_{j+1} \lmydots] 
	\end{align*}
	which simplifies using $P(x,0)=1$ to:
\begin{align*}
	T_4' + T_8' & =  \sum_{k \le j} (-1)^{k+j + \sum_{m=1}^{k-1} n_m }   [\rmydots n_{k-1}, 1, n_k, n_{j-2} \mydots n_{j-1}, n_{j+1} \lmydots] \\
	T_5' & =   \sum_{k \le j} (-1)^{k+j+1 +   \sum_{m=1}^{k-1} n_m} [ \rmydots n_{k-1}, 1, n_k \mydots n_{j-1}, n_{j+1}, \lmydots] 
\end{align*}
Hence we have 
\[ 	(D - 2\delta \circ S - 2 S \circ \delta )[n_1, \dots, n_s] = 2T_4' + 2T_5' + 2T_6' = 0, \]
 if $n_j = 2$ with $j < s$. In the case $n_s=2$, we get
	\begin{align*} &	(D - 2\delta \circ S - 2S \circ \delta )[n_1 \mydots n_{s-1}, 2] \\  & = 2T_4' + 2T_8' \\ & = 2 \sum_{1 \le k \le s} (-1)^{s+ k+  \sum_{m=1}^{k-1} n_m } [n_1 \mydots n_{k-1}, 1, n_k \mydots n_{s-1}].  \end{align*}
A similar argument deals with the case that some $n_j=1$.
\end{proof}

\begin{lem}
	For every partition $[n_1, \dots, n_s]$ with all $n_i$ even we have 
	\[ ( D - 2\delta \circ S - 2 S \circ \delta) \Ins_t [n_1, \dots, n_{s-1}, 2] = 2 (t+1)(-1)^{t+1} \Ins_{k+1} [n_1, \dots, n_{s-1}] .\] 
\end{lem}

\begin{proof} Take any $I \subset \{1, \dots, s+t \}$ with $|I| = t+1$. The coefficient of the term $ \Ins_I [n_1 \mydots n_{s-1}]$ in  $(D- 2\delta \circ S - 2 S \circ \delta) \Ins_t[n_1 \mydots n_{s-1}, 2]$ is given by 
	\[ 2 (-1)^{st+t} \sum_{i \in I} (-1)^{i + \sum_{j \in I, j <i} 1 + \sum_{j \in I, j < i} j + \sum_{j \in I, j>i} (j-1)} = 2 (-1)^{s(t+1)} (t+1) (-1)^{\sum_{j \in I} j} . \]
	This is the coefficient of $ \Ins_I [n_1 \mydots n_{s-1}]$ in $2 (t+1) (-1)^{t+1} \Ins_{t+1} [n_1, \dots, n_{s-1}]$.
\end{proof}

\begin{cor} \label{E2} Let $p > 2$. Take $x_1^{c_1} \cdots x_k^{c_k}  y_1^{d_1} \cdots y_l^{d_l} y_0$ with size $m$. Then
	\[ (D - 2 \delta \circ S - 2 S \circ \delta)( x_1^{c_1} \cdots x_k^{c_k} y_1^{d_1} \cdots y_l^{d_l}) = 0 \] 
and
	\[ (D - 2 \delta \circ S - 2 S \circ \delta)( x_1^{c_1} \cdots x_k^{c_k} y_1^{d_1} \cdots y_l^{d_l} y_0) = 2 (-1)^{n-m+1} (n-m+1) x_1^{c_1} \cdots x_k^{c_k} y_1^{d_1} \cdots y_l^{d_l}. \]
\end{cor}

\begin{cor} \label{E3} Let $p=2$. Take $x_1^{c_1} \cdots x_k^{c_k}  y_1^{d_1} \cdots y_l^{d_l} y_0$ with size $m$. Then
	\[ (D - 2 \delta \circ S - 2 S \circ \delta)( x_2^{c_2} \cdots x_k^{c_k} y_1^{d_1} \cdots y_l^{d_l} ) = 0 \] 
	and if $c_1>0$
\[ (D - 2 \delta \circ S - 2 S \circ \delta)( x_1^{c_1} \cdots x_k^{c_k} y_1^{d_1} \cdots y_l^{d_l} ) = 2 (-1)^{n-m+3} (n-m+3) x_1^{c_1-1} \cdots x_k^{c_k} y_1^{d_1} \cdots y_l^{d_l} y_0. \]
Furthermore,
	\[ (D - 2 \delta \circ S - 2 S \circ \delta)( x_1^{c_1} \cdots x_k^{c_k} y_1^{d_1} \cdots y_l^{d_l} y_0) = 2 (-1)^{n-m+1} (n-m+1) x_1^{c_1} \cdots x_k^{c_k} y_1^{d_1} \cdots y_l^{d_l}. \]
\end{cor}

This allows us to compute the map $D^* \colon H^i(A_n^\bullet) \to H^{i-1}(A_{n-1}^\bullet)$ with both $\ZZ$ and $\FP$-coefficients.

\section{Proof of Main Theorem}

	\begin{proof}[Proof of Th. \ref{mainfp}] By lemma \ref{E1} and corollary \ref{E2} we can conclude that the rank of the map $D^* \colon H^r(A_n^\bullet, \FP) \to H^{r-1}(A_{n-1}^\bullet, \FP)$ is given by the number of monomials \[ x_1^{c_1} \dots x_k^{c_k} y_0 y_1^{d_1} \dots y_l^{d_l} \] of degree $r$ and size $m \le n$ such that $p \nmid 2(n-m+1)$. Equivalently, the rank can be written as 
	\[ B'_p(n, r) = \left | \left \{ \begin{matrix} 1 \le a_1 \le a_2 \le \dots \le a_g \\   1 \le b_1 < b_2 < \dots < b_h \end{matrix} \middle |  \begin{matrix} 2\sum_i p^{a_i}+ 2\sum_j p^{b_j}+1 -  2g-h  = r \\ 2\sum_i p^{a_i}+ 2\sum_j p^{b_j}+2  \le n \\ p \nmid  2(n- 2\sum_i p^{a_i} - 2\sum_j p^{b_j} -1)  \end{matrix} \right \} \right | \] 
	By the long exact sequence of lemma \ref{longexact} we have determined
\[ \dim H^r(C_n(S^2), \FP) = B_p(n,r)+ B_p(n-1, r-2) - B'_p(n, r)- B'_p(n, r-1).\]
\end{proof}

	\begin{cor} \cite{salvatore}
		This can be written as a generating series. Let \[Q = \prod_{i>0}\frac{1+w^{2p^i-1}z^{2p^i}}{1-w^{2p^i-2}z^{2p^i}}.\] Then we have for $p>2$:
		\[ \sum_{r, n \ge 0} \dim H^r(C_n(S^2), \FP) \,  w^r z^n= \left ( \frac{1}{1-z} + \frac{w z^{p+1}}{1-z^p} + \frac{w^3 z^3}{1-z} + \frac{w^2 z}{1-z^p} \right ) Q\]
\end{cor}

\begin{cor}
	Our description implies eventual periodicity \[ \dim H^r(C_{n+p}(S^2), \FP) = \dim H^r(C_{n}(S^2), \FP) \] if $n \ge 2r$. 
\end{cor}

\begin{proof} 
	As $ \sum_{i=1}^g p^{a_i} + \sum_{j=1}^h p^{b_j} \ge 2g +h$, we get the inequalities $r \ge 2g +h+1$ and  $\sum_{i=1}^g p^{a_i} + \sum_{j=1}^h p^{b_j} \le 2 r-2$.
	Hence we have for $n \ge 2r+2$:
	\begin{align*} B_p(n, r) = B_p(n+1, r) & &  B'_p(n+p, r)= B'_p(n,r)  \end{align*}
\end{proof}

	\begin{proof}[Proof of Th. \ref{mainzz}]

		For $n \le 3$, we can easily check the theorem by hand. Take $n \ge 4$. We look at the beginning of the long exact sequence of lemma \ref{longexact}. We immediately read off
		\[ H^0(B_n^\bullet) \simeq H^0(\An).\]	As $H^2(\An) = H^2(A_{n-1}^\bullet) = 0$ by application of lemma \ref{mainvain}, we get the exact sequence 
	\begin{align*} 0 \to H^1(B_n^\bullet) \to H^1(\An) \xrightarrow{D^*} H^0(A_{n-1}) \to H^2(B_n^\bullet) \to 0.\end{align*}
		The group $H^1(\An) = \ZZ$ is generated by the class of $y_0$ and the group $H^0(A_{n-1}^\bullet) = \ZZ$ is generated by the class $1$ with the map $D^*(y_0) = (2n-2) \cdot 1$ by lemma \ref{E2}. Hence we see \begin{align*} H^1(B_n^\bullet) = 0  & & H^2(B_n^\bullet) = \ZZ / (2n-2) \ZZ.\end{align*}

If we had $D = 2 \delta \circ S + 2 S \circ \delta$, we would have a chain map
\[ A_n^\bullet \to B_n^\bullet, \, \, a \mapsto (a,- \, 2(-1)^{n-r} S(a)),  \]
that would split the sequence
	\[ 0 \to A_{n-1}^\bullet [2] \to B_n^\bullet \to A_n^\bullet \to 0, \, \, a_2 \mapsto (0, a_2), \, (a_1, a_2) \mapsto a_1\]
on the right.

	In our case, the long exact sequence of lemma \ref{longexact}  gives us short exact sequences
	\[0 \to \Coker D^* \to H^r(B_n^\bullet) \to \Ker D^* \to 0. \]
		We want to construct a right splitting $s: \Ker D^* \to H^r(B_n^\bullet)$. For $r \ge 2$, the cohomology group $H^r(\An)$ is finite and has no elements of order $p^2$. For every prime $p$, we can take a $\FP$-basis of the $p$-torsion in $\Ker D^*$ consisting of the classes $\overline{b_i}$ of the chains \[b_i = \tilde{\beta} (m_i) = \frac{1}{p} \delta(m_i)\] for some monomials $m_i = x_1^{a_1} \dots x_k^{a_k}y_1^{b_1} \dots y_l^{b_l}  y_0^{b_0} \in \An$. By corollary \ref{E2}, we can find integers $k_i$ and monomials $m_i'$ such that 
\[ (D -  2 S \circ \delta - 2 S \delta \circ S )(m_i) = k_i p m_i'. \] If $p \neq 2$ and $y_0 \mid m_i$, we have $m_i' = x_1^{a_1} \dots x_k^{a_k} y_1^{b_1} \dots y_l^{b_l}$. 
Define $E = D - 2 S \circ \delta - 2 \delta \circ S$. Observe that $ E \circ S = S \circ E$. Hence we get
		\begin{align*} E(m_i) = p k_i m_i'  & & E( b_i) = k_i \delta(m_i'). \end{align*}

	
	Define a map \[ s \colon \Ker D^* \to H^r(B_n^\bullet, \ZZ) \] by setting 
	\[ s(\bar{b_i}) =  \left (b_i, \,-  2 (-1)^{n-r} S(b_i)- (-1)^{n-r} k_i m_i' \right ).\] We see that
	\begin{align*} 
		\Delta \circ s( \bar{b_i}) & =  \left (\delta (b_i), \, -2 (-1)^{n-r} \delta \circ S (b_i) + (-1)^{n-r} D(b_i) - (-1)^{n-r} k_i \delta (m'_i ) \right ) \\
		 & = 		\left (\delta (b_i), \, 2 (-1)^{n-r} S  \circ \delta(b_i)+ (-1)^{n-r}  E (b_j) -(-1)^{n-r} k_i \delta(m'_i) \right ) \\
	& 	= 0 \end{align*} and hence $s(\bar{b_i})$ is a cycle in $H^r(B_n^\bullet, \ZZ)$. We have to show that $p s(\bar{b_i})$ is a boundary. We have $p b_i = \delta(m_i)$ and can compute
	\begin{align*} 
p s(\bar{b_i}) & = \left (p b_i, \,  -2 (-1)^{n-r} S(pb_i) - (-1)^{n-r} p k_i m'_i \right ) \\ 
		& = \left (\delta (m_i), \, -2 (-1)^{n-r} S \circ \delta(m_i) - (-1)^{n-r} p k_i m'_i \right ) \\ 
		& = \left (\delta (m_i), \, (-1)^{n-r}( 2 \delta \circ S(m_i) -  D(m_i) +  E(m_i)-  p k_i m'_i) \right ) \\
		& = \left (\delta (m_i), \,  2 (-1)^{n-r} \delta \circ S(m_i) - (-1)^{n-r}  D(m_i) \right ) \\
		& = \Delta \left ( m_i, \, S(m_i) \right ).
	\end{align*}
	Hence $s$ is a well-defined right splitting of the sequence
	\[0 \to \Coker D^* \to H^r(B_n^\bullet) \to \Ker D^* \to 0. \]
		For $r \ge 3$, both $\Ker D^*$ and $\Coker D^*$ have no elements of $p^2$, thus the same is true for $H^r(B_n^\bullet)$.
\end{proof}

\begin{example}
	We want to compute the 3-torsion in the groups $H^6(C_9(S^2), \ZZ)$ and $H^6(C_{10}(S^2), \ZZ)$. We use the long exact sequence
	\[ \dots \to H^{5}(A_{n}^\bullet) \xrightarrow{ D^*}  H^{4}(A_{n-1}^\bullet) \rightarrow H^6(B_{n}^\bullet) \rightarrow H^6(A_{9}^\bullet) \xrightarrow{D^*} H^{5} (A_{n-1}^\bullet) \rightarrow \dots \]

For $p=3$, the generators of $H^*(A_n^\bullet, \ZZ_3)$ are:
	\begin{figure}[h]
	\begin{tabular}{c|cccccc}
		generator & $x_1$ & $x_2$  & $y_0$ & $y_1$ & $y_2$ & \dots \\
		\midrule
		degree & 4 & 16 & 1 & 5 & 17 & \dots \\
		size & 6 & 18 & 2 & 6 & 18 & \dots \\
	\end{tabular}
	\end{figure}
\end{example}
So \[ H^6(A_{9}^\bullet, \ZZ_3) = H^6(A_{10}^\bullet, \ZZ_3) = \ZZ_3  y_0 y_1.\]
and \[ H^4(A_9^\bullet, \ZZ_3) = H^4(A_{10}^\bullet, \ZZ_3) = \ZZ_3 x_1 \]

We have $D^*(y_0 y_1) = 2(n-7) y_1$ and $D^*(x_1 y_0) = 2 (n- 7)x_1$. Hence we get
\begin{align*} H^6(B_9^\bullet, \ZZ_3) =  0 & & H^6(B_{10}^\bullet, \ZZ_3) = \ZZ_3^2. \end{align*}
The Bockstein $\tilde{\beta}(x_1y_0) = y_0 y_1$ shows 
 \[ H^6(A_9^\bullet, \ZZ)_{(3)} = H^6(A_{10}^\bullet, \ZZ)_{(3)} = \ZZ_3  y_0 y_1\]
and \[ H^4(A_9^\bullet, \ZZ)_{(3)} = H^4(A_{10}^\bullet, \ZZ)_{(3)} = 0. \]
We get 
\begin{align*} H^6(B_9^\bullet, \ZZ)_{(3)} =  0 & & H^6(B_{10}^\bullet, \ZZ)_{(3)} = \ZZ_3. \end{align*}

%
%
%
%

\section{Some Tables}

The tables \ref{cczz} and \ref{s2zz} were computed with the help of the computer algebra systems Sage \cite{sage} and Magma \cite{magma}. The cohomology groups $H^r(C_n(S^2), \ZZ)$ have already been determined for $n \le 9$ by Sevryuk \cite{sevryuk} and Napolitano \cite{napolitano}.
%
%
%
%
%
%
%
\begin{landscape}
\begin{table}[h]
	\caption{ Cohomology groups $H^i(C_n(\CC), \ZZ)$}
	\label{cczz}
	\small
	\begin{tabular}{c|cccccccccccccccc}
		\diagbox{$n$}{$i$}	& 0 & 1 & 2 & 3 & 4 & 5 & 6 & 7 & 8 & 9 & 10 & 11 & 12 & 13 & 14 & 15 \\
		\hline
		$1$ & $\ZZ$ \\
		$2,3  $& $\ZZ $&$  \ZZ $& \\
		$4,5$ & $\ZZ $&$  \ZZ $& 0 & $\ZZ_2$ \\
		$6,7$ & $\ZZ $&$  \ZZ $& 0 & $\ZZ_2$ & $\ZZ_2$ & $\ZZ_3$ \\
		$8,9$ & $\ZZ $&$  \ZZ $& 0 & $\ZZ_2$ & $\ZZ_2$ & $\ZZ_6$ & $\ZZ_3$ & $\ZZ_2$ \\
		$10, 11$ & $\ZZ $&$  \ZZ $& 0 & $\ZZ_2$ & $\ZZ_2$ & $\ZZ_6$ & $\ZZ_6$ & $\ZZ_2$ & $\ZZ_2$ & $\ZZ_5$ \\
		$12, 13$ & $\ZZ $&$  \ZZ $& 0 & $\ZZ_2$ & $\ZZ_2$ & $\ZZ_6$ & $\ZZ_6$ & $\ZZ_2^2$  & $\ZZ_2$ & $\ZZ_2 \plh \ZZ_3 \plh \ZZ_5$ & $\ZZ_2 \plh \ZZ_5$ \\
		$14, 15$ & $\ZZ $&$  \ZZ $& 0 & $\ZZ_2$ & $\ZZ_2$ & $\ZZ_6$ & $\ZZ_6$ & $\ZZ_2^2$  & $\ZZ_2^2$ & $\ZZ_2 \plh \ZZ_3 \plh \ZZ_5$ & $\ZZ_2^2 \plh \ZZ_3 \plh \ZZ_5$ & $\ZZ_2$ & $0$ & $\ZZ_7$ \\
		$16, 17$ & $\ZZ $&$  \ZZ $& 0 & $\ZZ_2$ & $\ZZ_2$ & $\ZZ_6$ & $\ZZ_6$ & $\ZZ_2^2$  & $\ZZ_2^2$ & $\ZZ_2^2 \plh \ZZ_3 \plh \ZZ_5$ & $\ZZ_2^2 \plh \ZZ_3 \plh \ZZ_5$ & $\ZZ_2^2$ &$\ZZ_2$ & $\ZZ_2 \plh \ZZ_7$ & $\ZZ_7$ & $\ZZ_2$ \\
	\end{tabular}
	\vspace{0.5cm}
	\caption{ Cohomology groups $H^i(C_n(S^2), \ZZ)$}
	\label{s2zz}
	\small
	\begin{tabular}{c|cccccccccccccccc}
		\diagbox{$n$}{$i$}	& 0 & 1 & 2 & 3 & 4 & 5 & 6 & 7 & 8 & 9 & 10 & 11 & 12 & 13 & 14 & 15 \\
		\hline
		$1$ & $\ZZ$ & $0$ & $\ZZ$ \\
		$2$ & $\ZZ$ & $0$ & $\ZZ_2$ \\
		$3  $& $\ZZ $&$  0 $&$   \ZZ_4 $&$  \ZZ $\\
		$4 $&$   \ZZ $ & 0 & $ \ZZ_6$  & $\ZZ \plh \ZZ_2$  \\
		$5 $&$   \ZZ $&$ 0 $&$  \ZZ_8 $&$ \ZZ \plh \ZZ_2 $&$ 0  $&$ \ZZ_2$ \\
		$6 $&$   \ZZ $&$ 0 $&$  \ZZ_{10} $&$ \ZZ \plh \ZZ_2 $&$ \ZZ_2 $&$ \ZZ_2 \plh \ZZ_3     $& \\
		$7 $&$   \ZZ $&$ 0 $&$  \ZZ_{12} $&$ \ZZ \plh \ZZ_2 $&$ \ZZ_2 $&$ \ZZ_2 \plh \ZZ_3      $&$ \ZZ_2      $&$ \ZZ_3$ \\
		$8 $&$   \ZZ $&$ 0 $&$  \ZZ_{14} $&$ \ZZ \plh \ZZ_2 $&$ \ZZ_2 $&$ \ZZ_2^2 \plh \ZZ_3 $&$ \ZZ_2      $&$ \ZZ_2           $& \\
		$9 $&$   \ZZ $&$ 0 $&$  \ZZ_{16} $&$ \ZZ \plh \ZZ_2 $&$ \ZZ_2 $&$ \ZZ_2^2 \plh \ZZ_3 $&$ \ZZ_2      $&$ \ZZ_2^2      $&$ \ZZ_3           $&$ \ZZ_2$ \\
		$10 $&$  \ZZ $&$ 0 $&$  \ZZ_{18} $&$ \ZZ \plh \ZZ_2 $&$ \ZZ_2 $&$ \ZZ_2^2 \plh \ZZ_3 $&$ \ZZ_2^2 \plh \ZZ_3 $&$ \ZZ_2^2 \plh \ZZ_3      $&$ \ZZ_2 \plh \ZZ_3           $&$ \ZZ_2 \plh \ZZ_{5}                $& \\
		$11 $&$   \ZZ $&$ 0 $&$  \ZZ_{20} $&$ \ZZ \plh \ZZ_2 $&$ \ZZ_2 $&$ \ZZ_2^2 \plh \ZZ_3 $&$ \ZZ_2^2 $&$ \ZZ_2^2      $&$ \ZZ_2^2 \plh \ZZ_3      $&$ \ZZ_2 \plh \ZZ_5                $&$ \ZZ_2                $&$ \ZZ_5$ \\
		$12 $&$   \ZZ $&$ 0 $&$  \ZZ_{22} $&$ \ZZ \plh \ZZ_2 $&$ \ZZ_2 $&$ \ZZ_2^2 \plh \ZZ_3 $&$ \ZZ_2^2 $&$ \ZZ_2^3  $&$ \ZZ_2^2 \plh \ZZ_3      $&$ \ZZ_2^2 \plh \ZZ_3 \plh \ZZ_5      $&$ \ZZ_2^2 $  & $0$       \\
		$13  $&$  \ZZ $&$ 0 $&$  \ZZ_{24} $&$ \ZZ \plh \ZZ_2 $&$ \ZZ_2 $&$ \ZZ_2^2 \plh \ZZ_3 $&$ \ZZ_2^2 \plh \ZZ_3 $&$ \ZZ_2^3 \plh\ZZ_3 $&$ \ZZ_2^2 \plh \ZZ_3      $&$ \ZZ_2^3 \plh \ZZ_3 \plh \ZZ_5   $&$ \ZZ_2^2           $&$ \ZZ_2 \plh \ZZ_3            $&$ \ZZ_2 \plh \ZZ_5           $&\\
		$14  $&$  \ZZ $&$ 0 $&$  \ZZ_{26} $&$ \ZZ \plh \ZZ_2 $&$ \ZZ_2 $&$ \ZZ_2^2 \plh \ZZ_3 $&$ \ZZ_2^2 $&$ \ZZ_2^3 $&$ \ZZ_2^3 \plh \ZZ_3 $&$ \ZZ_2^3 \plh \ZZ_3 \plh \ZZ_5      $&$ \ZZ_2^3      $&$ \ZZ_2^2        $&$ \ZZ_2 \plh \ZZ_5          $&$ \ZZ_7       $& \\
		$15 $&$    \ZZ $&$ 0 $&$  \ZZ_{28} $&$ \ZZ \plh \ZZ_2 $&$ \ZZ_2 $&$ \ZZ_2^2 \plh \ZZ_3 $&$ \ZZ_2^2 $&$ \ZZ_2^3 $&$ \ZZ_2^3 \plh \ZZ_3 $&$ \ZZ_2^3 \plh \ZZ_3 \plh \ZZ_5 $&$ \ZZ_2^4 $&$ \ZZ_2^2      $&$ \ZZ_2^2 \plh \ZZ_3 \plh \ZZ_5      $&$ \ZZ_2 \plh \ZZ_7     $&$ 0 $&$ \ZZ_7$\\
		$16 $&$   \ZZ $&$ 0 $&$  \ZZ_{30} $&$ \ZZ \plh \ZZ_2 $&$ \ZZ_2 $&$ \ZZ_2^2 \plh \ZZ_3 $&$ \ZZ_2^2 \plh \ZZ_3 $&$ \ZZ_2^3 \plh \ZZ_3 $&$ \ZZ_2^3 \plh \ZZ_3 $&$ \ZZ_2^4 \plh \ZZ_3 \plh \ZZ_5 $&$ \ZZ_2^4 \plh \ZZ_3 \plh \ZZ_5 $&$ \ZZ_2^3 \plh \ZZ_3 \plh \ZZ_5 $&$ \ZZ_2^3 \plh \ZZ_3 \plh \ZZ_5 $&$ \ZZ_2^2 \plh \ZZ_7 $&$ 0 $&$ \ZZ_2 $ \\
\end{tabular}
\end{table}
\end{landscape}



  \bibliographystyle{alpha}
 \bibliography{thesis}
\end{document}